\documentclass[12pt, reqno]{amsart}
\usepackage{amsmath, amsfonts, amssymb}
\usepackage{eucal}
\usepackage{latexsym}
\usepackage{cite}
\usepackage{fancyhdr}
\fancypagestyle{firststyle}
{
   \fancyhf{}
   
   \fancyhead[R]{\large J. Amer. Math. Soc. 26 (2013) 831–-852}
}
\title{The logarithmic Minkowski problem}

\author{K\'aroly J. B\"or\"oczky}
\address{Alfr\'ed R\'enyi Institute of Mathematics
 Hungarian Academy of Sciences}
\email{carlos@renyi.hu}
\author{Erwin Lutwak}
\address{Polytechnic Institute of New York University,
Brooklyn, New York}
\email{elutwak@poly.edu}
\author{Deane Yang}\address{Polytechnic Institute of New York University,
Brooklyn, New York}
\email{dyang@poly.edu}
\author{Gaoyong Zhang}
\address{Polytechnic Institute of New York University,
Brooklyn, New York}
\email{gzhang@poly.edu}

\thanks{Research of the first author supported, in part, by EU FP7 IEF
grant GEOSUMSET and OTKA 075016. Research of the other three authors
supported, in part, by
NSF Grant DMS-1007347.}

\newcommand{\R}{\mathbb{R}}

\newcommand{\rn}{\mathbb R^n}
\newcommand{\sn}{S^{n-1}}

\newcommand{\psum}{{+_{\negthinspace\kern-2pt p}}\,}
\newcommand{\dpsum}{{\tilde+_{\negthinspace\kern-1pt p}}\,}
\newcommand{\lsub}[1]{\hskip -1.5pt\lower.5ex\hbox{$_{#1}$}}

\newcommand{\cK}{{\mathcal K}}

\newcommand{\cH}{{\mathcal H}}

\newcommand{\dimension}{\operatorname{dim}}

\newcommand{\relinterior}{\operatorname{relint}}
\newcommand{\SL}{\operatorname{SL}}
\newcommand{\GL}{\operatorname{GL}}
\newcommand{\piper}{\operatorname{P}\negthinspace\negthinspace}
\newcommand{\ess}{\operatorname{S}}

\begin{document}

\begin{abstract}
In analogy with the classical Minkowski problem, necessary and
sufficient conditions are given to assure that a given measure on the
unit sphere is the cone-volume measure of the unit ball of a finite
dimensional Banach space.
\end{abstract}

\date{\today}

\maketitle
\thispagestyle{firststyle}

\numberwithin{equation}{section}

\newtheorem{theo}{Theorem}[section]
\newtheorem{coro}[theo]{Corollary}
\newtheorem{lemma}[theo]{Lemma}
\newtheorem{prop}[theo]{Proposition}
\newtheorem{conj}[theo]{Conjecture}
\newtheorem{example}[theo]{Example}

\theoremstyle{definition}
\newtheorem*{defi}{Definition}

\section{Introduction}

The setting for this paper is $n$-dimensional Euclidean space, $\rn$.
A {\it convex body} in $\rn$ is a compact convex set that has
non-empty interior.
A {\it polytope} in $\rn$ is the convex hull of a finite set of points
in $\rn$ provided it has  positive {\it volume} (i.e., $n$-dimensional
volume).  The convex hull of a subset of these points is
called a {\it face} of the polytope if it lies entirely on the
boundary of the polytope and if it has  positive {\it area} (i.e.,
$(n-1)$-dimensional volume).

One of the cornerstones of the Brunn-Minkowski theory of convex bodies
is the {\it Minkowski problem}. It can be stated in a simple way for polytopes:
\vskip 6pt

\noindent
{\bf Discrete Minkowski problem.}
Find necessary and sufficient conditions on a set of unit vectors
$u_1,\ldots,u_m$ in $\rn$ and a set of real numbers
$a_1,\ldots, a_m>0$ that will guarantee the existence of an $m$-faced
polytope in $\rn$ whose faces have
outer unit normals  $u_1,\ldots,u_m$ and corresponding
face-areas  $a_1,\ldots, a_m$.
\vskip 6pt

More than a century ago, this problem was completely solved by Minkowski
himself \cite{Mink1897}: If the unit
vectors do not lie on a closed hemisphere of $\sn$, then a solution to
the Minkowski problem
exists if and only if
$$a_1 u_1 + \cdots + a_m u_m = 0.$$
In addition, the solution is unique up to a translation.

The discrete Minkowski problem prescribes the areas of faces of
a polytope. A natural, but still unsolved, problem involves
prescribing the cone-volumes of the polytope. If a polytope contains
the origin in its interior, then the {\it cone-volume}
associated with a face of the polytope
is the volume of the convex hull of the face and the origin.

\vskip 6pt
\noindent
{\bf Discrete logarithmic Minkowski problem.}
Find necessary and sufficient conditions on a set of unit vectors
$u_1,\ldots,u_m$ in $\rn$ and a set of real numbers $v_1, \ldots, v_m>0$
that guarantee the existence of an $m$-faced
polytope, that contains the origin in its interior, whose faces have
outer unit normals $u_1,\ldots,u_m$ and whose corresponding
cone-volumes are $v_1, \ldots, v_m$.
\vskip 6pt

To state the logarithmic Minkowski problem for general convex bodies,
we need to define the {\it cone-volume measure} of a convex body.
If $K$ is a convex body in $\rn$ that contains the origin in its
interior, then the
cone-volume measure, $V_K$, of $K$ is a Borel measure on the unit sphere
$\sn$ defined for a Borel set $\omega \subset \sn$, by
\begin{equation}
V_K(\omega) =\frac1n \int_{x\in\nu_K^{-1}(\omega)}
x\cdot \nu_K(x) \, d\mathcal H^{n-1}(x),
\end{equation}
where
$\nu_K : \partial'\negthinspace K \to \sn$  is the Gauss map of $K$, defined on
$ \partial'\negthinspace K$, the set of points of  $ \partial K$ that
have a unique outer unit normal, and $\mathcal H^{n-1}$ is
$(n-1)$-dimensional Hausdorff measure. In recent years,
cone-volume measures have appeared  in e.g.\ \cite{BGMN8},
\cite{GM18}, \cite{Lud10}, \cite{LR10}, \cite{N45}, \cite{N46},
\cite{PW12}, and \cite{Sta12}.

If $K$ is a polytope whose outer unit normals are
$u_1,\ldots,u_m$ and whose corresponding cone-volumes are $v_1,
\ldots, v_m$, then the
cone-volume measure of $K$ is the discrete measure
\[
V_K = v_1 \delta_{u_1} + \cdots + v_m \delta_{u_m},
\]
where $\delta_{u_i}$ denotes the delta measure that is concentrated on $u_i$.

\vskip 6pt
\noindent
{\bf Logarithmic Minkowski problem.}
Find necessary and sufficient conditions on a finite Borel
measure $\mu$ on the unit sphere $\sn$ so that $\mu$ is the
cone-volume measure of a convex body in $\rn$.
\vskip 6pt

The associated partial differential equation for the logarithmic
Minkowski problem is the following Monge-Ampere type equation on the
sphere: Given $f:\sn \to (0,\infty)$, solve
\begin{equation}\label{PDE}
h \det(h_{ij} + h\delta_{ij}) = f,
\end{equation}
where $h_{ij}$ is the covariant derivative of $h$ with respect to
an orthonormal frame on $S^{n-1}$ and $\delta_{ij}$ is the Kronecker delta.
Here the cone-volume measure $\mu$ is assumed to have
the density $f$, with respect to spherical Lebesgue measure. In
connection with \eqref{PDE}, we
observe a new phenomenon: the solution for measure data does not
follow from the solution for
function data, but requires a fundamentally new approach. This will be
explained after the
formulation of Theorem \ref{main-theorem}.

In \cite{Lut93jdg}, the second named author introduced
the notion of $L_p$-surface area measure and posed the associated
$L_p$-Minkowski problem which has the classical Minkowski problem
and the logarithmic Minkowski problem as two important cases.

If $K$ is a convex body in $\rn$ that contains the origin in its interior
and $p\in \mathbb R$, then the $L_p$-surface area measure $S_p(K,
\cdot)$ of $K$ is a Borel measure on the unit sphere
$\sn$ defined  on a Borel $\omega \subset \sn$, by
\begin{equation}\label{p-surface-measure}
S_p(K, \omega) = \int_{x\in\nu_K^{-1}(\omega)}
\big(x\cdot \nu_K(x)\big)^{1-p} \, d\mathcal H^{n-1}(x).
\end{equation}

When $p=1$, the measure $S_1(K,\cdot)$ is the classical surface area
measure of $K$.
When $p=0$, the measure $\frac1n S_0(K,\cdot)$ is the cone-volume
measure of $K$. When $p=2$, the measure $S_2(K,\cdot)$ is called the
quadratic surface area measure of $K$, which was studied in
\cite{LutYanZha00duke},
\cite{LutYanZha02duke}, and \cite{Lud03}. Applications of the $L_p$-surface area
measure to affine isoperimetric inequalities were given in e.g.,
\cite{CamGro02adv},
\cite{LutYanZha00jdg},
and \cite{LutYanZha05jdg}.

The following $L_p$-Minkowski problem is one of the central problems
in contemporary convex geometric analysis.

\vskip 6pt
\noindent
{\bf $L_p$-Minkowski problem.}
Find necessary and sufficient conditions on a finite Borel
measure $\mu$ on the unit sphere $\sn$ so that $\mu$ is the
$L_p$-surface area measure of a convex body in $\rn$.
\vskip 6pt

The case $p=1$ of
the $L_p$-Minkowski problem  is of course the
classical Minkowski problem, which was solved by Minkowski,
Alexandrov, and Fenchel and Jessen (see Schneider \cite{Sch93} for
references). Landmark contributions to establishing regularity
for the Minkowski problem are due to (among others) Lewy \cite{Lew38},
Nirenberg \cite{Nir53},
Cheng and Yau \cite{ChengYau}, Pogorelov \cite{Pogorelovbook},
and Caffarelli \cite{Caf90Ann}.

For $p>1$,
a solution to the $L_p$-Minkowski problem was given
in \cite{Lut93jdg} under the assumption
that $p\neq n$ and that measure $\mu$ is {\it even} (assumes the same
values on antipodal Borel subsets of $S^{n-1}$) . In
\cite{LutYanZha04tams}, it was shown that, for
$p\neq n$, the $L_p$-Minkowski problem has an
equivalent volume-normalized formulation, and (using a slightly
modified approach) a solution of the {\it even
volume-normalized} $L_p$-Minkowski problem was given for all $p>1$.
The regular even $L_p$-Minkowski problem was studied in
\cite{LutOli95}. Extensions of the $L_p$-Minkowski problem are
studied in \cite{HLYZ10}.

In the plane ($n=2$), the $L_p$-Minkowski problem was treated by
Stancu \cite{Sta02, Sta03},  Umanskiy \cite{Uma03}, Chen
\cite{Chen06adv}, and by Jiang \cite{Jiang}.
Solutions to the $L_p$-Minkowski problem
are the homothetic solutions of Gauss curvature flows
(see e.g., \cite{And99inv}, \cite{And03jams},
\cite{Chou85Com}, \cite{Chow85jdg}, \cite{GagHam86jdg}).
When the measure $\mu$ is proportional to Lebesgue measure
on the unit circle, $S^1$, solutions to the $L_p$-Minkowski problem in
$\mathbb R^2$
are the homothetic solutions of isotropic curve flows classified
by Andrews \cite{And03jams}.

The $L_p$-Minkowski problem (without the assumption that the data is
even) was treated by Guan and Lin \cite{GuaLin04} and by Chou
and Wang \cite{ChoWan06}. Hug et al.\  \cite{HLYZ05dcg} gave an
alternate approach to some of the results of Chou and Wang
\cite{ChoWan06}.

The solution of the even $L_p$-Minkowski problem was a critical
ingredient that allowed the authors of \cite{LutYanZha02jdg} to
extend the affine Sobolev inequality \cite{Zhang99} and obtain
the $L_p$ affine Sobolev inequality, and later enabled Cianchi et al.\
\cite{CLYZ} to establish the affine Moser-Trudinger and the affine
Morrey-Sobolev inequalities. These were then strengthened by Haberl et
al in \cite{HS09jdg,HS09jfa,HSX11ma}. Connections of the $L_p$-Minkowski
problem with optimal Sobolev norms was shown in \cite{LutYanZha06}.

There are other important works on extensions and analogues of the
Minkowski problem,
see e.g., \cite{ColFim10},  \cite{GuaGua02Ann}, \cite{GuaMa03Inv},
\cite{HuMaShe04}, and \cite{Jer96acta}.

Much of the past work on the $L_p$-Minkowski problem is limited to
the case $p>1$. One of the reasons is that the uniqueness of
the $L_p$-Minkowski problem for $p>1$ can be shown by using
mixed volume inequalities (see \cite{Lut93jdg}).
When $p<1$, the problem becomes challenging because there are no mixed
volume inequalities available as of yet.
The case $p=0$, called the logarithmic case, is probably the
most important case with geometric significance because it is
the singular case. The cone-volume measure is the only one
among all the $L_p$-surface area measures  \eqref{p-surface-measure}
that is $\SL(n)$ invariant; i.e., for $\phi \in \SL(n)$,
\[
S_0(\phi^t K, \omega) = S_0(K, \langle{\phi \omega}\rangle),
\]
where $\langle{\phi\omega}\rangle =\{ \phi u / |\phi u| : u\in \omega\}$.
In light of the equivalence of the study of finite dimensional Banach spaces
and that of origin-symmetric convex bodies and the fact that the
cone-volume measure
of an origin-symmetric convex body is even, the following even
logarithmic Minkowski problem is of great interest.

\vskip 6pt
\noindent
{\bf Even logarithmic Minkowski problem.}
Find necessary and sufficient conditions on an
even Borel measure $\mu$ on the unit sphere $\sn$ so that
$\mu$ is the cone-volume measure of an origin-symmetric convex
body in $\rn$.
\vskip 6pt

It is the aim of this paper to solve the existence part of the even
logarithmic Minkowski problem.

\begin{defi}
A finite Borel measure $\mu$ on $S^{n-1}$ is said to satisfy the
{\it subspace concentration inequality} if, for every subspace $\xi$ of
$\R^n$, such that $0<\dimension \xi<n$,
\begin{equation}\label{properdef0}
 \mu(\xi\cap S^{n-1})\leq \frac 1n \, \mu(S^{n-1})\dimension \xi.
\end{equation}
The measure is said to satisfy the {\it subspace concentration
condition}  if in addition to satisfying
the subspace concentration inequality (\ref{properdef0}), whenever
\begin{equation*}
 \mu(\xi\cap S^{n-1}) = \frac 1n \, \mu(S^{n-1})\dimension \xi,
\end{equation*}
for some subspace $\xi$, then
there exists a subspace
$\xi'$, that is  complementary to $\xi$ in $\R^n$,  so that also
\begin{equation*}
 \mu(\xi'\cap S^{n-1})= \frac 1n \, \mu(S^{n-1})\dimension \xi',
\end{equation*}
or equivalently so that $\mu$ is concentrated
on $S^{n-1}\cap(\xi\cup \xi')$.

The measure $\mu$ on $S^{n-1}$ is said to satisfy the {\it
strict subspace concentration inequality}
if the inequality in (\ref{properdef0}) is strict for each subspace $\xi
\subset \rn$, such that $0<\dimension \xi<n$.
\end{defi}

The following theorem is the main result of this paper.
It gives the necessary and sufficient conditions for the existence
of a solution to the even logarithmic Minkowski problem.

\begin{theo}\label{main-theorem}
A non-zero  finite even Borel measure on the unit sphere $\sn$ is
the cone-volume measure of an origin-symmetric convex body
in $\rn$ if and only if it satisfies the subspace concentration
condition.
\end{theo}

We are dealing with the general case where the data is a measure
(which of course includes as a special case the {\it discrete
log-Minkowski problem} where the measure is concentrated on a finite
set of points). This is much harder than the case where the given data
$\mu$ in (\ref{PDE}) is a function. It is remarkable that the subspace
concentration condition, which is satisfied by all cone-volume
measures of convex bodies, is also the critical and only condition
that is needed for existence. For functions, the subspace concentration
condition is trivially satisfied but for measures it is precisely
what is necessary.

We need to observe a crucial difference between the log-Minkowski
problem and the classical
Minkowski problem, in fact, all the cases of the $L_p$-Minkowski problem
where $p\geq1$.  Once the $L_p$-Minkowski problem has been solved (for
any particular value of $ p \geq 1$) for the case where the data
consists of functions, the general  $L_p$-Minkowski problem (where the
data is given by measures) can be solved by an approximation argument.
However, the solution to the log-Minkowski problem for the general
case (i.e., for measures) does not follow from its solution for the
cases where the data is given by functions --- at least not by any
approximation argument known to the authors.

The discrete planar case of Theorem~\ref{main-theorem}
was proved by Stancu \cite{Sta02}.

Uniqueness for the logarithmic Minkowski problem is not treated
in this paper. In determining the ultimate shape of his worn stone,
Firey \cite{Fir74} showed that if the cone-volume
measure of a smooth origin-symmetric convex body in $\rn$ is a
constant multiple of the
Lebesgue measure on $\sn$, then the convex body must be a ball.
Firey conjectured that his symmetry assumption was necessitated only
by his methods.
In $\mathbb R^3$, Firey's conjecture was established by
Andrews \cite{And99inv}.

Firey's theorem regarding unique determination by cone-volume measures can be
shown to hold for measures other than Lebesgue measures. This will be
treated in a separate paper.

\section{Preliminaries}

We develop some notation and, for quick later reference,  list some basic facts
about convex bodies. Good general references for the theory of convex
bodies are provided by the books of Gardner \cite{Gar95}, Gruber
\cite{Gruberbook}, Schneider \cite{Sch93}, and Thompson
\cite{Tho96}.

The standard inner product of the vectors $x,y\in\R^n$ is denoted by
$x\cdot y$.  We write $|x|^2=x\cdot x$, and  $\sn=\{x\in\R^n:\,\,
|x|=1\}$ for the
boundary of the Euclidean unit ball $B$ in $\R^n$.  The volume of $B$
will be denoted by $\omega_n$.

The set of continuous functions on the sphere $\sn$ will be
denoted by $C(\sn)$ and will always be viewed as equipped with the
max-norm metric:
\[ |f-g|_{\infty}=\max_{u\in\sn}|f(u)-g(u)|, \]
for $f,g\in C(\sn)$.
The set of strictly positive continuous functions will be
denoted by $C^+(\sn)$, and $C^+_e(\sn)$ will denote the subset of
$C^+(\sn)$ consisting of only
the even functions.

Write $V_i$ for the $i$-dimensional Lebesgue measure. When $i=n$,
the subscript will be suppressed, and $V_n$ will be simply written as $V$.  For
$k$-dimensional Hausdorff measure, we write $\cH^k$. The letter $\mu$
will be used
exclusively to denote a finite Borel measure on $\sn$. For such a
measure $\mu$, we denote by $|\mu|$ its total mass, i.e.
$|\mu|=\mu(\sn)$.

We write $o$ for the origin of $\R^n$,
and $[x_1,\ldots,x_i]$ to denote the convex hull
of the points $x_1,\ldots,x_i \in \R^n$.
For a non-zero $u \in \R^n$ or a linear subspace $\xi$, let $u^\bot$
and $\xi^\bot$, denote
the orthogonal complement of the respective linear subspace. Moreover, write
$\piper_\xi:\R^n\to \xi$ for the orthogonal projection onto $\xi$.

A {\it convex body} is a compact convex subset of $\R^n$
with non-empty interior. The set of
convex bodies in $\R^n$ containing the origin in their interiors is
denoted by $\cK^n_o$. The set of convex bodies in $\R^n$ that are
symmetric about the origin will be denoted by $\cK^n_e$.
If  $\xi \subset \R^n$ is an affine subspace, and $K$ is a convex body
in $\xi$, then the set of relative interior points of $K$
with respect to $\xi$ is denoted by $\relinterior K$.

The  {\it support function} $h_K:\R^n\to \R$
of a compact, convex $K\subset \R^n$ is defined, for $x\in\mathbb{R}^n$, by
\[
 h_K(x)=\max\{x\cdot y:\,\,y\in K\}.
\]
Note that support functions are positively homogeneous of degree one
and subadditive. From the definition, it follows immediately that, for
$\phi \in \GL(n)$, the support function of $\phi K =\{\phi x:\,\, x\in
K\}$ is given by
\[
h_{\phi K}(x) = h_K(\phi^t x),
\]
for $x\in\rn$.
The support function of a body $K\in\cK^n_o$ is
strictly positive and continuous on the unit sphere $\sn$.
If $x\in\partial'\negthinspace K$, then the {\it supporting distance}
of $\partial K$ at $x$ is defined to be
$x\cdot\nu_K(x)=h_K(\nu_K(x))$ and will be denoted by $d_K(x)$.

The set $\cK^n_o$ will be viewed as equipped with the Hausdorff
metric and thus for a sequence $\{K_i\}$ of bodies in $\cK^n_o$
and a body $K\in \cK^n_o$, we have
$\lim_{i\to\infty}K_i=K$ provided that
\[
| h_{K_i}-h_{K} |_{\infty} \to0.
\]

A boundary point $x\in\partial K$ is said to have $u\in\sn$ as an
outer normal provided $x\cdot u =h_K(u)$. A boundary point is
said to be {\it singular} if it has more than one unit normal
vector. It is well known (see, e.g., \cite{Sch93}) that the set of
singular boundary points of a convex body has
$\cH^{n-1}$-measure equal to
$0$.

For each Borel set $\omega\subset\sn$, the
{\it inverse spherical image} of $\omega$
is the set of all points of $\partial K$ which have an outer unit
normal belonging to the set $\omega$. Since the inverse spherical
image of $\omega$ differs from $\nu_K^{-1}(\omega)$  by a set of
$\cH^{n-1}$-measure equal to $0$, we will often make no distinction
between the two sets.

Associated with each convex
body $K \in \cK^n_o$ is a Borel measure $S_K$ on $S^{n-1}$ called
the {\it Aleksandrov-Fenchel-Jessen surface area measure} or the {\it
surface area measure} of
$K$, defined for each Borel set $\omega \subset S^{n-1}$ as the
$\cH^{n-1}$-measure of the inverse spherical image of $\omega$, or
equivalently
\begin{equation}\label{defsurfaceareameasure}
S_K(\omega)=\cH^{n-1}(\nu_K^{-1}(\omega)).
\end{equation}
From \eqref{p-surface-measure}
and \eqref{defsurfaceareameasure}, we see that $S_1(K,\cdot) = S_K$.
We shall require the basic volume formula
\begin{equation}\label{volrep}
V(K)=\frac1n\int_{u\in\sn}h_K(u) \,dS_K(u).
\end{equation}
As is well known, and easily shown, the measure $S_p(K,\cdot)$ is
absolutely continuous with respect to the measure $S_K$, and its
Radon-Nikodym derivative is $h_K^{1-p} $; i.e.,
\begin{equation}
dS_p(K,\cdot) = h_K^{1-p} dS_K.
\end{equation}
Obviously, the measure $S_p(K,\cdot)$ is homogeneous of degree $n-p$
with respect to dilation of $K$, that is,
$S_p(\lambda K, \cdot)=\lambda^{n-p} S_p(K,\cdot)$, for $\lambda >0$.

We will make use of the weak continuity of surface
area measures; i.e., if $\{K_i\}$ is a sequence of bodies in
$\cK^n_o$ then
\begin{equation}\label{weakconv}
\lim_{i\to\infty} K_i=K\in\cK^n_o\quad\Longrightarrow\quad
\lim_{i\to\infty} S_{K_i}=S_K,\,\,\textnormal{weakly}.
\end{equation}

If $M$ is a convex body in $\rn$ and $x_0\in M$, then
 $\rho_{-x_0+M}: \rn \setminus \{0\}\to [0,\infty)$, the {\it radial
function} of $M$ with respect to $x_0$, is defined for $x\in \rn
\setminus \{0\}$ by
\[
\rho_{-x_0+M}(x)=\max\{\lambda\ge0 : x_0 + \lambda x\in M\}.
\]
Note that for $u\in\sn$,  the distance from $x_0$ to  $\partial M$ in
direction $u$ is precisely $\rho_{-x_0+M}(u)$, that is, $x_0 +
\rho_{-x_0+M}(u)u \in \partial M$.

\section{The variational method}

A function $h\in C^+(\sn)$ defines a family $\{H_u\}_{u\in\sn}$ of
hyperplanes
\[H_u=\{x\in\R^n:\,\,x\cdot u=h(u)\}.\]

Consider the intersection of the halfspaces that
are associated to $h$ and bounded by the family $\{H_u\}_{u\in\sn}$. This gives
rise to the convex body
\[
K=\  \bigcap_{u\in\sn}\{x\in\R^n:\,\,x\cdot u\leq h(u)\}.
\]
The body $K$ is called the {\it Aleksandrov body} (also known as the
{\it Wulff shape})
associated with $h$. Note that since $h$ is both strictly positive and
continuous its Aleksandrov body, $K$, must be an element of
$\cK^n_o$. The Aleksandrov body associated with $h$ can
alternatively be defined as the unique maximal element, with respect
to set inclusion, of the set
\[\{Q\in\cK^n_o:\,\, h_Q\leq h\}.\]
Obviously, for the Aleksandrov body $K$ associated with the function
$h$, we have
$$
h_K \le h,
$$
and it turns out that, in fact,
\begin{equation}\label{wulfft}
h_{K}= h,\qquad\textnormal{a.e. with respect to} \,\,S_{K}.
\end{equation}

If $h$ is the support function of a
convex body $K\in\cK^n_o$, then $K$ itself is the Aleksandrov body
associated with $h$. If $h$ is an even
function, then the Aleksandrov body associated with $h$ is
origin-symmetric. We will need  {\it Aleksandrov's
convergence lemma} (see, e.g., \cite[Lemma 6.5.2]{Sch93}): If the
functions $h_i\in C^+(\sn)$ have
associated Aleksandrov bodies $K_i\in\cK^n_o$, then
\[h_i \to h \in C^+(\sn) \qquad
\Longrightarrow\qquad K_i \to K,\] where $K$ is the Aleksandrov body
associated with $h$.

The volume $V(h)$ of a function $h\in C^+(\sn)$ is defined as the
volume of the Aleksandrov body associated with $h$. Since the
Aleksandrov body associated with the support function $h_K$ of a
convex body $K\in\cK^n_o$ is the body $K$ itself, we have
\begin{equation}\label{volext}
V(h_K)=V(K).
\end{equation}
Obviously,  the functional $V:C^+(\sn)\to (0,\infty)$ is homogeneous
of degree $n$; i.e., for $f\in C^+(\sn)$ and real $s>0$,
\begin{equation}\label{homog1}
V(sf) = s^n V(f).
\end{equation}
From Aleksandrov's convergence lemma and the continuity of ordinary volume on
$\cK^n_o$ we see that
\begin{equation}\label{cont1}
V:C^+(\sn)\to (0,\infty)\quad\text{is continuous}.
\end{equation}

Let $I\subset\R$ be an interval containing 0 and suppose
$h_t(u)=h(t,u):I\times \sn\to(0,\infty)$ is continuous. For fixed
$t\in I$, let
\[
K_t=\ \bigcap_{u\in\sn} \{ x \in\R^n:\,\, x\cdot u\leq h_t(u)\}
\]
be the Aleksandrov body associated with $h_t$. The family of convex bodies
$K_t$ will be called {\it the family of Aleksandrov
bodies associated with $h_t$}. Obviously, we can rewrite \eqref{wulfft} as
\begin{equation}\label{wulff1}
h_{K_t}\leq h_t\qquad\text{and}\qquad h_{K_t}= h_t,\
\textnormal{a.e. with respect to} \,\,S_{K_t},
\end{equation}
for each $t\in I$.

The following form (proved in e.g., \cite{HLYZ10}) of Aleksandrov's
Lemma (see e.g.,  \cite{Sch93})  will be needed.

\begin{lemma} \label{var-formula-wulff}
Suppose $I\subset\R$ is an open interval containing $0$, and that the
function $h_t=h(t,u):I\times
\sn\to(0,\infty)$ is continuous. If, as $t\to 0$, the convergence in
\[
\frac{h_t -h_0}{t}\ \to\ f=\left.\frac{\partial h_t}{\partial t}
\right\vert_{t=0},\
\]
is uniform on $\sn$, and if $K_t$ denotes the
Aleksandrov body associated with $h_t$, then
\[\lim_{t\to 0}\frac{V(K_t)-V(K_0)}{t}=\int_{\sn}f\,dS_{K_0}.\]
\end{lemma}

\section{A minimization problem}

Let $\mu$ be a finite even Borel measure on $\sn$
with total mass $|\mu|>0$.
Define the logarithmic functional $\Phi_\mu : \mathcal K^n_e \to \R$,
\begin{equation}
\Phi_\mu(K) = \int_{\sn} \log h_K \, d\mu.
\end{equation}

Consider the minimization problem,
\begin{equation}\label{min-problem}
\inf\{ \Phi_\mu(Q) :  \text{$V(Q)=|\mu|$ and $Q \in \mathcal K^n_e$} \}.
\end{equation}

\begin{lemma}\label{minimization-existence}
Let $\mu$ be a finite even Borel measure on $\sn$
with $|\mu|>0$. If $K_0\in  \mathcal K^n_e$ is an origin-symmetric
convex body such that
$V(K_0)=|\mu|$ and
$$
 \Phi_\mu(K_0) =\inf\{ \Phi_\mu(Q) :  \text{$V(Q)=|\mu|$ and $Q \in
\mathcal K^n_e$} \},
$$
then the measure $\mu$ is the cone-volume measure of $K_0$.
\end{lemma}

\begin{proof}
It is easily seen that it is sufficient to establish the Lemma under
the assumption that $\mu$ is a probability measure.

Define the functional
$M_0:C_e^+(\sn)\to (0,\infty)$ by
\[
M_0(q) = \frac1{ V(q)^{1/n} }    \exp\left(\int_{S^{n-1}}
\log q\, d\mu\right),
\]
for $q\in C_e^+(\sn)$.
Since the functional $V:C_e^+(\sn)\to (0,\infty)$
is continuous, we see that the functional $M_0:C_e^+(\sn)\to
(0,\infty)$ is continuous as well.
From (\ref{homog1}) we see that the functional $M_0:C_e^+(\sn)\to
(0,\infty)$ is homogeneous of degree $0$; i.e.,
 for $q\in C_e^+(\sn)$ and real $s>0$, we have $M_0(sq) = M_0(q)$.

Consider the minimization problem,
\begin{equation}\label{min-problem-function}
\inf\{M_0(q) : q\in C_e^+(\sn)\}.
\end{equation}
Suppose $f\in C_e^+(S^{n-1})$. Let $K$ be the Aleksandrov
body associated with $f$. Then $V(f)=V(h_K)=V(K)$ but $h_K\leq f$. Therefore,
$M_0(h_K) \leq  M_0(f)$.  We can therefore limit our search for the infimum
of $M_0$ by restricting our attention to support functions of
origin-symmetric convex
bodies. Since $M_0$ is homogeneous of degree $0$, the
infimum of $M_0$ is
\[
\inf \{ M_0(q) : q\in C_e^+(S^{n-1}) \} =
\inf \{ e^ { \Phi_\mu(Q)}:  \text{$V(Q)=1$
and $Q \in \mathcal K^n_e$} \}.
\]
The hypothesis of our Lemma is that the right infimum is in fact a
minimum and that it is attained at
$K_0\in \mathcal K^n_e$. Therefore, the support function $h_{K_0}>0$ is
a solution
of the minimization problem (\ref{min-problem-function}); i.e.,
\begin{equation}\label{proof1}
\inf \{ M_0(q) : q\in C_e^+(S^{n-1}) \} =M_0(h_{K_0}).
\end{equation}

Suppose $g\in C_e(S^{n-1})$ is arbitrary but fixed.
Consider the family $h_t\in C_e^+(S^{n-1})$, where the function
$h_t=h(t,\cdot):\R\times\sn \to (0,\infty)$ is defined by
\[
h_t=h(t,\cdot) = h_{K_0} e^{ t g},
\]
and let $K_t$ denote the Aleksandrov body associated with $h_t$
(since $h_0$ is the support function of the convex body $K_0$ our
notation is consistent).

Since $g$ is bounded on $\sn$, for $h_t = h_{K_0} e^{ t g}$, the
hypothesis of Lemma \ref{var-formula-wulff}
\[
\frac{h_t -h_0}{t}\ \to\ g h_{K_0},\qquad\text{uniformly on $\sn$,}
\]
as $t\to 0$, is satisfied and we get from Lemma \ref{var-formula-wulff}
\begin{equation}\label{proof3}
\left.\frac{d}{dt} V(K_t)\right\vert_{t=0}\  =\  \int_{\sn} g h_{K_0}\,dS_{K_0}.
\end{equation}
Now (\ref{proof3})  shows that the function $t\mapsto M_0(h_t)$, where
\begin{equation}\label{proof4}
M_0(h_t) =  V(K_t)^{-1/n}     \exp\left(\int_{S^{n-1}}
\log(h_{K_0} e^{ t g})\, d\mu\right),
\end{equation}
is differentiable at  $t=0$, and, after recalling that  by hypothesis
$V(K_0)=1$,  and using (\ref{proof3}), differentiating in
(\ref{proof4}) gives us
\begin{equation}\label{proof2}
\left.\frac{d}{dt} M_0(h_t)\right\vert_{t=0} =
 \left[-\frac1{n}\int_{S^{n-1}} g h_{K_0}\, dS_{K_0}\  +\
\int_{S^{n-1}} g\, d\mu\right] {\exp\left(\int_{S^{n-1}} \log  h_{K_0}\,
d\mu\right)}.
\end{equation}
But (\ref{proof1}) shows that the function $t\mapsto M_0(h_t)$ has a
minimum at $t=0$ which gives
\[\left.\frac{d}{dt} M_0(h_t)\right\vert_{t=0} = 0,
\]
and now (\ref{proof2}) allows us to conclude that
\[
\frac1{n}\int_{S^{n-1}} g h_{K_0} \, dS_{K_0}\  =\  \int_{S^{n-1}} g\, d\mu.
\]
But since this must hold for arbitrary $g\in C_e(\sn)$,  we conclude
\[
d\mu = \frac1n h_{K_0}  dS_{K_0},
\]
as desired.
\end{proof}

\section{Necessary conditions for existence}

We shall use the method of symmetrization to prove the necessity part
of Theorem \ref{main-theorem}.
Consider a proper subspace $\xi \subset \rn$ and let $m=\dimension  \xi$.
We consider the symmetrization of the convex body $K\in \mathcal K^n_o$ with
respect to $\xi$.  For each  $x\in \piper_\xi K$, we replace
$K\cap(x+\xi^\bot)$ with the
 $(n-m)$-dimensional ball $B^{n-m}_{r(x)}(x)$ in $x+\xi^\perp$ that is
centered at $x$ and whose radius $r(x)$ is chosen so that
 $\omega_{n-m}r(x)^{n-m}=V_{n-m}(K\cap(x+\xi^\bot))$; i.e., $r(x)$ is
chosen so that $B^{n-m}_{r(x)}(x)$ has
the same $(n-m)$-dimensional volume as $K\cap(x+\xi^\bot)$. Denote the new body
by ${\ess}_\xi K$. That is,
\[
{\ess}_\xi K = \bigcup_{x\in \piper_\xi K} B^{n-m}_{r(x)}(x), \qquad
\omega_{n-m}r(x)^{n-m}=V_{n-m}(K\cap(x+\xi^\bot)).
\]
As is well known (see e.g. \cite{GarZha}), the Brunn-Minkowski
inequality (see the beautiful survey \cite{Gar02}) implies that
${\ess}_\xi K$ is
also a convex body. By Fubini's theorem, in fact, by Cavalieri's principle,
$V(K)=V({\ess}_\xi K)$. If $K$ is origin-symmetric, then
${\ess}_\xi K$ is also origin-symmetric.

For notational simplicity, denote ${\ess}_\xi K$
by $\tilde K$, and denote the
Gauss map of $K$ by $\nu$ and that of $\tilde K$ by $\tilde \nu$.
Also abbreviate $B^{n-m}_{r(x)}(x)$ by $B(x)$. If $C\subset \R^n$ is
convex and compact, then for notational simplicity, write
$\partial C$ for the relative boundary of $C$, with respect to the
affine hull of $C$.

\begin{lemma}\label{cone-measure-symmetrization}
Suppose $\xi  \subset \rn$ is a subspace of $\R^n$ such that
$0<\dimension \xi<n$, and suppose $K\in \mathcal K^n_o$.
Then the cone-volume measures of $K$ and
the symmetrization, $\tilde K$, of $K$ about $\xi $ satisfy
\begin{equation*}
V_{\tilde K} (\xi \cap \sn) = V_K (\xi \cap \sn).
\end{equation*}
\end{lemma}

\begin{proof} Let $m=\dimension \xi$ and let $S^{m-1}=\xi \cap S^{n-1}$.

For $z\in  \R^n$,  write $z=(x,y)$,  where
$x=\piper_\xi  z$ and $y=\piper_{\xi^\bot}z$.  We will identify $x$
with $(x,0)$.
If $z\in\nu^{-1}(S^{m-1})$,  then $x\in \partial (\piper_\xi K)$.
But for $x\in \partial (\piper_\xi K)$ and $u\in S^{m-1}=\xi \cap
S^{n-1}$, we obviously have $z\cdot u = x\cdot u$ because
$z\in x+\xi ^\bot$. This shows that $z\cdot \nu(z)$ is independent of
$y \in K\cap(x+\xi^\perp)$
for $z\in \nu^{-1}(S^{m-1})$. But $z\cdot \nu(z)$ is the supporting distance of
$\partial K$ at $z$ as well as $d_{\piper_\xi K}(x)$, the supporting
distance of $\partial
(\piper_\xi K)$ at $x$
for $\mathcal H^{m-1}$-almost all $x$ on $\partial (\piper_\xi K)$.
Similarly, for $z \in \tilde \nu^{-1}(S^{m-1})$,
we also have $z\cdot\tilde \nu(z)= d_{\piper_\xi\tilde K}(x)=
d_{\piper_\xi K}(x)$ because
$\piper_\xi K = \piper_\xi \tilde K$.
This, together with $V_{n-m}(B(x)) = V_{n-m}(K\cap(x+\xi^\bot))$ and
$\piper_\xi K = \piper_\xi\tilde K$,
gives
\begin{align*}
V_K(S^{m-1})
&=\frac1n \int_{z\in\nu^{-1}(S^{m-1})} z\cdot \nu(z)\, d\mathcal H^{n-1}(z)\\
&=\frac1n \int_{x\in\partial(\piper_\xi K)} d_{\piper_\xi
K}(x)V_{n-m}(K\cap(x+\xi^\bot))\, d\mathcal H^{m-1}(x) \\
&=\frac1n \int_{x\in\partial(\piper_\xi\tilde K)} d_{\piper_\xi\tilde
K}(x) V_{n-m}(B(x))\, d\mathcal H^{m-1}(x) \\
&=\frac1n \int_{ z\in \tilde\nu^{-1}(S^{m-1})}
z\cdot\tilde\nu(z)\,d\mathcal H^{n-1}(z) \\
&=V_{\tilde K}(S^{m-1}).
\end{align*}

\end{proof}

The necessity part of Theorem \ref{main-theorem}
is presented in Theorem \ref{necessity} below.
For origin-symmetric polytopes, the subspace concentration inequality
of Theorem \ref{necessity} was
established previously by He, Leng and Li \cite{HLL06},
with an alternate proof given by Xiong \cite{Xio}.  We now give
a new proof that is valid for arbitrary origin-symmetric convex bodies.
Applications of this theorem to reverse affine isoperimetric
inequalities were given in
\cite{HLL06} and \cite{Xio}.

\begin{theo}\label{necessity}
The cone-volume measure of an origin-symmetric convex body in $\rn$
satisfies the subspace concentration condition.
\end{theo}

\begin{proof}
Suppose $K$ is an origin-symmetric convex body in $\rn$. Let $\xi\subset
\rn$ be a subspace
such that $0<\dimension  \xi =m<n$, and let $S^{m-1}= \xi\cap
S^{n-1}$. Let $\tilde K$ be the
symmetrization of $K$ with
respect to $\xi$.

For $y\in B(o)=\piper_{\xi^\bot}\tilde K$, let $\rho_{-y+\tilde
K}=\rho_{\xi\cap (-y+\tilde K)}: S^{m-1} \to (0,\infty)$, be the
radial function of $\xi\cap (-y+\tilde
K)$.

There are three basic observations here.  First, for each
$u\in S^{m-1}$, the value of $\rho_{ -y+\tilde K}(u)$ is independent of $y\in
\piper_{\xi^\bot} B(\rho_{ \tilde K}(u)\,u)$. Then, that for each $z=(x,y)\in
\tilde\nu^{-1}(S^{m-1})$,  the value of  $z\cdot \tilde \nu(z)$
is independent of
$y \in (x+\xi^\perp)\cap\tilde  K =B(x)$. And
that $z\cdot \tilde\nu(z)$, the supporting distance of
$\partial\tilde K$ at $z$, is equal to $d_{\piper_\xi \tilde K}(x)$,
the supporting distance of $\partial
(\piper_\xi \tilde K)$ at $x$,
for $\mathcal H^{m-1}$-almost all $x$ on $\partial (\piper_\xi \tilde K)$.
From this and the definition of cone-volume measure
we have
\begin{eqnarray*}
V_{\tilde K}(S^{m-1})
&=&\frac1n\int_{z\in\tilde\nu^{-1}(S^{m-1})}
 z\cdot \tilde\nu(z) \, d{\mathcal H}^{n-1}(z)\\
&=&\frac1n\int_{x\in \partial (\piper_\xi \tilde K)}
d_{\piper_\xi \tilde K}(x) \, V_{n-m}\left(B(x)\right)\,d {\mathcal
H}^{m-1}(x)\\
&=&\frac1n\int_{u\in S^{m-1}}
\rho_{ \tilde K}(u)^m \, V_{n-m}\left(B(\rho_{ \tilde K}(u)u)\right)
\,d {\mathcal H}^{m-1}(u)\\
&=&\frac1n\int_{u\in S^{m-1}}\left(\int_{y\in
\piper_{\xi^\bot}B(\rho_{ \tilde K}(u)u)} \rho_{-y+\tilde K}(u)^m \,d
{\mathcal H}^{n-m}(y)\right)\,d {\mathcal H}^{m-1}(u),
\end{eqnarray*}
where in going from the second to the third line we changed variables
$x=\rho_{ \tilde K}(u)u$, for $u\in S^{m-1}$, and
used the fact that
\[
d_{\piper_\xi \tilde K}(x) \, d\mathcal H^{m-1}(x) = \rho_{ \tilde
K}(u)^m \, d\mathcal
H^{m-1}(u),
\]
(see e.g.,\cite{LutYanZha02duke}, Lemma~2).

Since $B(o)$ is maximal, $\tilde K$ is obviously the (disjoint) union
of the fibers $(y+\xi)\cap{\tilde K}$, with $y\in B(o)$. Thus,
\begin{eqnarray*}
V(\tilde K)
&=&\int_{y\in B(o)}V_m((y+\xi)\cap \tilde K)\,d {\mathcal H}^{n-m}(y)\\
&=&\int_{y\in B(o)}V_m(\xi\cap (-y+\tilde K))\,d {\mathcal H}^{n-m}(y)\\
&=&\frac1m\int_{y\in B(o)}\int_{u\in S^{m-1}}\rho_{-y+\tilde
K}(u)^m\,d {\mathcal
H}^{m-1}(u)\,d {\mathcal H}^{n-m}(y).
\end{eqnarray*}
Since $K$ is origin-symmetric, the ball $B(x)$ attains its maximum size when
 $x=o$. Thus, $\piper_{\xi^\bot}B(\rho_{\tilde K}(u)u) \subseteq
B(o)$, and we get
\[
V_{\tilde K}(S^{m-1}) \le \frac mn V(\tilde K),
\]
with equality if and only if  $\piper_{\xi^\bot}B(\rho_{\tilde K}(u)u)
= B(o)$, for all $u\in S^{m-1}$.
Since $V(K)=V(\tilde K)$, Lemma \ref{cone-measure-symmetrization} now
allows us to
conclude that
\begin{equation}\label{cone-subspace-inequality}
V_K(\xi\cap S^{n-1}) \le \frac mn \, V_K(\sn),
\end{equation}
that is,  the subspace concentration inequality holds for the
cone-volume measure of $K$.

To show that the cone-volume measure
satisfies the subspace concentration condition, we now analyze
the implications of equality in  \eqref{cone-subspace-inequality}. To that end,
assume that there is equality in (\ref{cone-subspace-inequality}).
Then $\piper_{\xi^\bot}B(\rho_{\tilde K}(u)u) = B(o)$, for all $u\in S^{m-1}$.
Suppose $x\in\piper_\xi K\setminus\{o\}$. The Brunn-Minkowski inequality yields
$\piper_{\xi^\bot}B(\rho_{\tilde K}(u)u)\subset \piper_{\xi^\bot} B(x)
\subset B(o)$
for $u=x/|x|$.
Therefore,
$\piper_{\xi^\bot}B(x)=B(o)$ for each $x\in\piper_\xi K$, and hence
$V_{n-m}((x+\xi^\bot)\cap K)$
is independent of  $x\in\piper_\xi K$. The equality conditions of the
Brunn-Minkowski
inequality tell us that this can only happen when $(x+\xi^\bot)\cap K$ is
a translate of  $\xi^\bot\cap K$ for each $x\in\piper_\xi K$.

Since  $K$ is origin symmetric, so is $\xi^\bot\cap K$.  Thus, for
each $x\in\piper_\xi K$ the body $(x+\xi^\bot)\cap K$  (being just a
translate of $\xi^\bot\cap K$) has a center of symmetry.
As $x$ traverses a line segment in $\piper_\xi K$,  the convexity
of $K$ guarantees that all the boundary points of  $(x+\xi^\bot)\cap K$ traverse
parallel line segments,
and thus the center of $(x+\xi^\bot)\cap K$ also traverses a parallel
line segment as well.
Therefore, as $x$ varies in $\piper_\xi K$, the center of
$(x+\xi^\bot)\cap K$
lies in an $m$-dimensional origin-symmetric convex compact set, which is
$$
C=\{z\in K : \,z+(\xi^\bot\cap K)\subset K\}.
$$
 Thus
 \[
 K=(\xi^\bot\cap K)+ C,
 \]
which shows that we can write $K$ as the Minkowski sum of an
$(n-m)$-dimensional and an $m$-dimensional convex set.

Let $\xi'$ be the orthogonal complement
of the linear hull of $C$. Now,
\[
\partial K \subseteq (\partial(\xi^\bot\cap K) + \relinterior
C)\cup(\relinterior (\xi^\bot\cap K) + \partial C)\cup
(\partial(\xi^\bot\cap K) + \partial C).
\]
But  $\nu(\partial(\xi^\bot\cap K) + \relinterior C)\subset \sn\cap\xi'$ and
$\nu(\relinterior (\xi^\bot\cap K) + \partial C)\subset \sn \cap\xi$, while
 $\cH^{n-1}(\partial(\xi^\bot\cap K) + \partial C)=0$.
Therefore, the cone-volume measure $V_K$
is concentrated on $\sn \cap (\xi\cup \xi')$.
\end{proof}

\section{Minimizing the logarithmic functional}

In this section, we prove that the necessity condition for cone-volume
measures is sufficient to
imply the existence of a solution of the minimization problem
(\ref{min-problem}).

We shall require the following trivial fact:

\begin{lemma}
\label{alphaomega}
If real $\alpha_1,\ldots,\alpha_n \ge 0$ satisfy
\begin{equation}\label{sumless}
\frac{\alpha_1+\cdots+\alpha_i}{i}<\frac1n,\text{\qquad for all
$i=1,\ldots,n-1$,}
\end{equation}
while
\begin{equation}\label{sumeq}
\alpha_1+\cdots+\alpha_n=1,
\end{equation}
then  there exists $t \in (0,1]$
such that for $\lambda =(1-t)/n,$ and
\begin{equation}\label{defbeta}
\beta_i=\alpha_i-\lambda,\  \text{for all  $i=1,\ldots,n-1$, while}
 \  \beta_n=\alpha_n - \lambda - t,
\end{equation}
we have
$$
\beta_1+\cdots+\beta_i \leq 0, \quad \text{for all $i=1,\ldots,n-1$,}
$$
and
$$
\beta_1+\cdots+\beta_n =0.
$$
\end{lemma}

\begin{proof}  Choose an $i_o  \in\{1,\ldots,n-1\}$ such that
$$
\frac{\alpha_1+\cdots+\alpha_i}{i} \le
\frac{\alpha_1+\cdots+\alpha_{i_o} }{i_o }  \qquad \text{for all
$i=1,\ldots,n-1$.}
$$
From (\ref{sumless}) we see there exists a $t\in (0,1]$ such that
\begin{eqnarray*}
\frac{\alpha_1+\cdots+\alpha_{i_o}}{i_o}  &=& (1-t)\frac1n=\lambda, \\
\frac{\alpha_1+\cdots+\alpha_i}i &\leq& (1-t) \frac1n=\lambda,
\qquad\text{for all  $i=1,\ldots,n-1.$}
\end{eqnarray*}
Then
\begin{align*}
\beta_1 + \cdots + \beta_i &= \alpha_1 + \cdots + \alpha_i - i\lambda
\le 0, \qquad
\text{for all  $i=1,\ldots,n-1,$}   \\
\beta_1 + \cdots + \beta_{n} &= \alpha_1 + \cdots + \alpha_{n} -
n\lambda -t = 1-n\lambda -t = 0.
\end{align*}
\end{proof}

\begin{lemma}\label{crosslimit}
Suppose $\mu$ is a probability measure  on $S^{n-1}$ that satisfies the strict
subspace concentration inequality. For each positive integer $l$,
let $u_{1,l},\ldots,u_{n,l}$ be an orthonormal basis of $\R^n$, and
suppose that $ h_{1,l}, \ldots, h_{n,l} $ are $n$ sequences of
positive real numbers
such that
$h_{1,l}\leq \cdots\leq h_{n,l}$, and such that the product
$h_{1,l}\cdots h_{n,l}\geq 1$,
and $\lim_{l\to\infty}h_{n,l}=\infty$.
Then, for the cross-polytopes
$Q_l=[\pm h_{1,l}u_{1,l},\ldots,\pm h_{n,l}u_{n,l}]$, the sequence
$$
\Phi_\mu(Q_l) = \int_{S^{n-1}} \log h_{Q_l}\, d\mu
$$
is not bounded from above.
\end{lemma}

\begin{proof}
For each $l$, we obviously have

\begin{equation}
\label{cross-support}
h_{Q_l}(v)=\max_{1\le i\le n}h_{i,l}|v \cdot u_{i,l}|,
\end{equation}
for $v\in S^{n-1}$.
After taking suitable subsequences,
we may conclude the existence of
an orthonormal basis $u_1,\ldots,u_n$
of $\R^n$ with

\begin{equation}\label{ulimit}
\lim_{l\to\infty}u_{i,l}=u_i\in S^{n-1}\qquad
\text{ for all $i=1,\ldots,n$}.
\end{equation}

For $i=1,\ldots,n$, let
$\xi_i$ be the subspace spanned by $\{u_1,\ldots,u_i\}$, and for convenience let
$\xi_{0}=\{o\}$. Since $\mu$ is a probability measure that
satisfies the strict subspace concentration inequality,
we have

\begin{equation}\label{muLi}
\sum_{j=1}^i \mu\left(S^{n-1}\cap (\xi_j\backslash \xi_{j-1}) \right)
=\mu\left(S^{n-1}\cap \xi_i \right)< \frac in,
\end{equation}
for $i=1,\ldots,n-1$.

Next we observe that for each $v\in S^{n-1}$, there exists a
$u_i$  such that $|v\cdot u_i| \geq 1/\sqrt{n}$. This simple observation
can be used to easily see that for each $\eta\in(0,1/\sqrt{n})$, the sets

\[
{A}_{i,\eta}=\{v\in S^{n-1}:\,
|v\cdot u_i|\geq \eta
\mbox{ \ and $|v\cdot u_j|<\eta$ for $j>i$}\}, \quad i=1,\ldots,n,
\]
 form a disjoint partition of  $S^{n-1}$.  Define

 \[
{B}_{i,\eta}=\{v\in S^{n-1}:\,
|v\cdot u_i|  >  0
\mbox{ \ and $|v\cdot u_j|<\eta$ for $j>i$}\}, \quad i=1,\ldots,n.
\]

 We claim that

\begin{equation}
\label{Xilimit}
\lim_{\eta\to 0^+}\mu({A}_{i,\eta})=
\mu\left(S^{n-1}\cap (\xi_i\backslash \xi_{i-1})\right).
\end{equation}
To see this, first note as $\eta$ decreases to $0$,
the sets ${A}_{i,\eta}\cap \xi_i$ form
an increasing family (with respect to set inclusion)
whose union (over $\eta$) is $S^{n-1}\cap(\xi_i\backslash \xi_{i-1})$.
It follows that

\[
\liminf_{\eta\to 0^+}\mu({A}_{i,\eta})\geq
\lim_{\eta\to 0^+}\mu({A}_{i,\eta}\cap \xi_i)=
\mu\left(S^{n-1}\cap (\xi_i\backslash \xi_{i-1})\right).
\]
Obviously, ${A}_{i,\eta}\subset {B}_{i,\eta}$.  As $\eta$ is
decreasing (to $0$), the sets
${B}_{i,\eta}$ form a decreasing family (with respect to set
inclusion) whose intersection (over $\eta$) is $S^{n-1}\cap (\xi_i\backslash
\xi_{i-1})$. Thus,

\begin{equation*}
\limsup_{\eta\to 0^+}\mu({A}_{i,\eta}) \leq
\lim_{\eta\to 0^+}\mu( {B}_{i,\eta})
=\mu(S^{n-1}\cap (\xi_i\backslash \xi_{i-1})),
\end{equation*}
which gives (\ref{Xilimit}).

Combining (\ref{Xilimit}) with (\ref{muLi})
shows that we may choose ${\eta_o}\in(0,1/\sqrt{n})$ small enough to satisfy

\begin{equation*}
\frac1i\sum_{j=1}^i \mu\left({A}_{j,{\eta_o}} \right)< \frac 1n,
\qquad\text{for all  $i=1,\ldots,n-1$.}
\end{equation*}
\medskip

But, since $\mu$ is a probability measure on $\sn$ and the ${A}_{j,{\eta_o}}$
form a disjoint partition
of $\sn$, we have

\[
\sum_{j=1}^n \mu({A}_{j,{\eta_o}})=1.
\]

Letting $\alpha_j=\mu({A}_{j,{\eta_o}})$, for $j=1,\ldots,n$,
Lemma~\ref{alphaomega} yields
a $t \in (0,1]$
such that for $\lambda =(1-t)/n,$ and
\begin{equation}\label{monika0}
\beta_i=\alpha_i-\lambda,\  \text{for all  $i=1,\ldots,n-1$,}\quad \text{while}
 \  \beta_n=\alpha_n - \lambda - t,
\end{equation}
we have
\begin{equation}\label{monika1}
\beta_1+\cdots+\beta_i \leq 0, \quad \text{for all $i=1,\ldots,n-1$,}
\end{equation}
and
\begin{equation}\label{monika2}
\beta_1+\cdots+\beta_n =0.
\end{equation}
(Note that ${\eta_o}$ and $t$ are positive
numbers independent of $l$.)
\smallskip

By  (\ref{ulimit}), we have $|u_{i,l}-u_i|<\frac {\eta_o} 2$, for all
$i=1,\ldots,n$
and sufficiently large $l$. Then for $u\in {A}_{i,{\eta_o}}$,
\begin{align*}
|u\cdot u_{i,l}| &\ge |u\cdot u_i| -|u\cdot (u_{i,l}-u_i)| \\
&\ge |u\cdot u_i| - |u_{i,l}-u_i| \\
&\ge \frac {\eta_o} 2,
\end{align*}
for all $i=1,\ldots,n$.
It follows from (\ref{cross-support}) that for $u\in{A}_{i,{\eta_o}}$,
\begin{equation}\label{cross-support-estimate}
 h_{Q_l}(u)\geq \frac {\eta_o} 2 h_{i,l},
\end{equation}
for all $i=1,\ldots,n$ and sufficiently large $l$.
\smallskip

It will be convenient to define $h_{n+1,l}=1$, for all $l$. For
sufficiently large $l$, using
the fact that the ${A}_{i,{\eta_o}}$ form a partition of $S^{n-1}$, followed
by \eqref{cross-support-estimate},
again using the fact that the ${A}_{i,{\eta_o}}$ form a partition of
$S^{n-1}$ and that $\mu$ is a probability measure, then
using \eqref{monika0},
then the fact that $\lambda$ is non-negative and the hypothesis that
$h_{1,l}\cdots h_{n,l}\geq 1$, for all $l$,
and finally from $0<h_{i,l}\leq h_{{i+1},l}$, for
$i=1,\ldots, n-1$, and \eqref{monika1} together with \eqref{monika2}, we get

\begin{eqnarray}
\nonumber
\int_{S^{n-1}} \log h_{Q_l}\,d\mu
&=&\sum_{i=1}^n\int_{{A}_{i,{\eta_o}}}\log h_{Q_l}\,d\mu \\
\nonumber
&\geq&
\log\frac{\eta_o}2 \sum_{i=1}^n\mu({A}_{i,{\eta_o}})
+\sum_{i=1}^n\mu({A}_{i,{\eta_o}})\log h_{i,l}\\
\nonumber
&=&
\log\frac{\eta_o}2+\sum_{i=1}^n\alpha_i\log h_{i,l}\\
\nonumber
&=& \log\frac{\eta_o}2+t\log h_{n,l}+\sum_{i=1}^n\lambda\log h_{i,l}+
\sum_{i=1}^n\beta_i\log h_{i,l}\\
\nonumber
&=& \log \left( \frac{\eta_o}2 h_{n,l}^t \right)+ \lambda\log (h_{1,l}
\cdots h_{n,l}) +
\sum_{i=1}^n\beta_i\log h_{i,l}\\
\nonumber
&\geq& \log \left(\frac{\eta_o}2 h_{n,l}^t\right)+
\sum_{i=1}^n\beta_i\log h_{i,l} \\
\nonumber
&=& \log\left(\frac{\eta_o}2 h_{n,l}^t\right)+
\sum_{i=1}^{n}(\beta_1+\cdots+\beta_i)
(\log h_{i,l}-\log h_{i+1,l}) \\
\label{int-theta-error}
&\geq& \log \left(\frac{\eta_o}2 h_{n,l}^t\right).
\end{eqnarray}
Since $t>0$, and by hypothesis $\lim_{l\to\infty}h_{n,l}= \infty$,
from \eqref{int-theta-error}
it follows that
\[
\lim_{l\to\infty}\int_{S^{n-1}}\log h_{Q_l}\,d\mu=\infty.
\]
\end{proof}

\begin{theo}\label{logmin0strict} Suppose $n\geq 2$, and $\mu$
is an even finite Borel measure on $S^{n-1}$ that
satisfies the strict subspace concentration inequality.  Then
there exists an origin-symmetric convex body $K \in \mathcal K^n_e$
such that
\[
\inf\left\{\int_{\sn} \log h_Q \, d\mu :  \text{$V(Q)=|\mu|$ and $Q \in
\mathcal K^n_e$} \right\}=\int_{\sn} \log h_K \, d\mu.
\]
\end{theo}

\begin{proof} Without loss of generality, assume that $\mu$ is
a probability measure.
Take a sequence  $Q_l\in{\mathcal K^n_e}$ such that $V(Q_l)=1$ and
\[
\lim_{l\to\infty}\Phi_\mu(Q_l)
=
\inf\{\Phi_\mu(Q) :\, \text{$V(Q)=1$ and $Q \in \mathcal K^n_e$}\}.
\]
Since the dilation, $\omega_n^{-1/n} B$, of the unit ball has unit
volume, and $\Phi_\mu(\omega_n^{-1/n}B)=-\frac1n\log \omega_n$, it follows
that
\begin{equation}\label{negative-inf}
\lim_{l\to\infty} \Phi_\mu(Q_l) \le -\frac1n\log \omega_n.
\end{equation}

By John's theorem \cite{Joh37}, there exists an ellipsoid $E_l$ centered
at the origin such that
$$
E_l\subset Q_l \subset \sqrt{n}\,E_l.
$$
Let $u_{1,l},\ldots,u_{n,l}\in S^{n-1}$ be the principal
directions of $E_l$ indexed to satisfy
$$
h_{1,l}\leq\cdots\leq h_{n,l}, \qquad\text{where
$h_{i,l}=h_{E_l}(u_{i,l})$,\quad for $i=1,\ldots,n$}.
$$
Next we define the cross-polytope
$$
C_l=[\pm h_{1,l}\,u_{1,l},\ldots,\pm h_{n,l}\,u_{n,l}].
$$
Since $C_l \subset  E_l \subset \sqrt{n} C_l$, we have
\begin{equation*}
C_l\subset Q_l\subset \,n \,C_l.
\end{equation*}
 We deduce from $V(Q_l)=1$  that $V(C_l) \ge n^{-n}$. Thus,
\begin{equation}
\label{hcond0}
 \prod_{i=1}^nh_{i,l}=
\frac{n!V(C_l)}{2^n}\geq \gamma,
\end{equation}
where $\gamma = \frac{n!}{2^n n^n}$.

Suppose that the sequence $\{Q_l\}$ is not bounded. Then $\{C_l\}$ is
not bounded, and thus, for a subsequence,
\begin{equation}\label{blow-up-condition}
\lim_{l\to \infty} h_{n,l} =\infty.
\end{equation}
In view of \eqref{hcond0} and \eqref{blow-up-condition},
applying Lemma~\ref{crosslimit} to $C'_l=\gamma^{\frac{-1}n}C_l$
yields that $\{\Phi_\mu(C'_l)\}$ is not bounded from above. Thus,
$\{\Phi_\mu(Q_l)\}$
is not bounded from above.
This contradicts \eqref{negative-inf}.
Therefore, the sequence $\{Q_l\}$ is bounded. By the Blaschke selection theorem,
$\{Q_l\}$ has a subsequence that converges to an origin-symmetric
convex body $K$.
Thus, the minimization problem (\ref{min-problem}) attains its minimum
at $K$.
\end{proof}

\section{Existence for the even logarithmic Minkowski problem}

Here we prove the sufficiency part of Theorem~\ref{main-theorem}.

\begin{lemma} \label{restriction}
Suppose $n\ge 2$ and $\mu$ is
a finite Borel measure on $S^{n-1}$
that satisfies the subspace concentration condition.
If $\xi$ is a subspace of $\R^n$ for which there is equality in
(\ref{properdef0}), that is
\begin{equation}\label{needname}
 \mu(\xi\cap S^{n-1}) = \frac 1n \, \mu(S^{n-1})\dimension \xi,
\end{equation}
then $\mu$ restricted to $S^{n-1}\cap \xi$ satisfies the subspace
concentration condition.
\end{lemma}
\begin{proof} Let $m=\dimension  \xi$, and let $S^{m-1}=\xi\cap
S^{n-1}$. Then (\ref{needname}) states that
\begin{equation}\label{needname2}
\frac1m\, {\mu(S^{m-1})}=\frac1n\,{{\mu(S^{n-1})}}.
\end{equation}
Suppose $\xi_k \subset \xi$ is a $k$-dimensional subspace. Then, from
the definition of $S^{m-1}$, the fact that $\xi_k\subset \xi$, the
fact that $\mu$ satisfies the subspace concentration condition, and
finally (\ref{needname2}), we have:
\begin{align*}
\mu(\xi_k\cap S^{m-1})&=
\mu(\xi_k\cap (\xi\cap\sn))\\
&= \mu(\xi_k\cap\sn) \\
&\le \frac kn \mu(\sn) \\
&=\frac 1{m} \mu(S^{m-1}) \dimension  \xi_k.
\end{align*}
This not only establishes the subspace concentration inequality for
$\mu$ restricted to $\xi\cap\sn$, but shows that equality in
\[
\mu(\xi_k\cap S^{m-1}) \le \frac 1{m} \mu(S^{m-1}) \dimension  \xi_k,
\]
implies that
\begin{equation}\label{needname3}
\mu(\xi_k\cap\sn)
= \frac kn \mu(\sn).
\end{equation}
But (\ref{needname3}) and the fact that $\mu$ satisfies the subspace
concentration condition yields the existence of a subspace $\xi_{n-k}$
complementary  to $\xi_k$ in $\R^n$ such that $\mu$ is concentrated on
$\sn\cap(\xi_k\cup \xi_{n-k})$. This implies that $\mu$
restricted to $\xi\cap\sn=S^{m-1}$ is concentrated on
$S^{m-1}\cap(\xi_k\cup(\xi_{n-k}\cap \xi))$.
\end{proof}

If $\mu$ is a Borel measure on $S^{n-1}$
and $\xi$ is a proper subspace of  $\R^n$, it will be convenient to write
$\mu_\xi$ for the restriction of $\mu$ to
$\sn\cap \xi$.

\begin{lemma}\label{last}
Let $\xi$ and $\xi'$ be complementary subspaces in $\R^n$ with
$0 < \dimension \xi =m < n$.
Suppose $\mu$ is an even Borel measure on $S^{n-1}$ that is concentrated on
$S^{n-1}\cap (\xi\cup \xi')$, and so that
\[
\mu(\xi\cap S^{n-1}) = \frac {m}n \mu(\sn).
\]
If $\mu_\xi$ and $\mu_{\xi'}$ are cone-volume measures of convex
bodies in the spaces $\xi$ and $\xi'$, then $\mu$ is
the cone-volume measure of a convex body in $\R^n$.
\end{lemma}

\begin{proof} By hypothesis
there exist convex bodies $C$ in $\xi$ and $C'$ in $\xi'$ so that
\[
V_C = a \mu_\xi, \quad V_{C'} = a \mu_{\xi'},
\]
where the positive constant $a$ will be chosen appropriately later.
Construct convex bodies $D$ in $\xi'^\perp$ and $D'$ in $\xi^\perp$ by
\begin{align*}
D&=\{(x+\xi^\perp)\cap \xi'^\perp : x\in C\}, \\
D'&=\{(x+\xi'^\perp)\cap \xi^\perp : x\in C'\}.
\end{align*}
Then $\piper_{\xi}D=C $ and $\piper_{\xi'}D'=C' $. Now,
\[
\frac{ V_m(D)}{ V_m(C)} = \frac{ V_{n-m}(D')}{ V_{n-m}(C')} = r,
\]
where $r$ can be viewed as the reciprocal of the cosine of the angle
between $\xi$ and $\xi'^\perp$. Then
\[
 V_m(D)=r  V_m(C) = rV_C(\sn\cap \xi) =ra \mu_{\xi}(\sn\cap \xi) =ra
\mu(\sn\cap \xi) = ra
\frac{m}n \mu(S^{n-1}).
\]
Similarly,
\[
 V_{n-m}(D')=ra \frac {n-m}n \mu(S^{n-1}).
\]

Observe that
\begin{equation}\label{sum-boundary}
\partial(D+D')=(\partial D +\relinterior D') \cup (\relinterior D + \partial D')
\cup (\partial D + \partial D').
\end{equation}

Consider $\rn$ as the orthogonal sum of $\xi$ and $\xi^\perp$ . Write
$y=(y_1,y_2)\in\rn$ and identify $y_1$ with $(y_1,0)$  and  $y_2$ with
$(0, y_2)$.
Obviously for each $y=(y_1,y_2)\in \partial D$,  we have ${y_1} \in
\partial C$ and  $y=({y_1}+\xi^\perp)\cap \xi'^\perp$.
For the outer normals on $\partial D +\relinterior D'$, we have
$\nu_{D+D'}(y+y')=\nu_{D+D'}(y)$.
These normals are orthogonal to $\xi^\perp \supset D'$ and thus belong
to $S^{n-1}\cap \xi$.
The normal $\nu_{D+D'}(y)$ is also orthogonal to the
$(m-1)$-dimensional support plane of $D$ at $y\in\partial D$ in
$\xi'^\perp$
and thus is orthogonal to the orthogonal projection of the support
plane of $D$ at $y$ onto $\xi$, that is orthogonal to the
$(m-1)$-dimensional support plane of $C$ at ${y_1}\in\partial C$ in
$\xi$.
It follows that $\nu_{D+D'}(y)=\nu_C({y_1})$.
We now see that for
$\cH^{m-1}$-almost all  $y\in \partial D$ and for all $y' \in
\relinterior D'$,
\begin{equation}\label{latest}
\text{$\nu_{D+D'}(y+y') = \nu_C({y_1})$\quad and\quad $ y\cdot \nu_C({y_1}) =
{y_1}\cdot \nu_C({y_1})$,}
\end{equation}
for  $\cH^{m-1}$-almost all  ${y_1}\in\partial C$.

Suppose $\omega \subset S^{n-1}\cap \xi$.  Obviously,
\[
\nu_{D+D'}^{-1}(\omega)\subset \partial D +\relinterior D'.
\]
From the definition of cone-volume measure, \eqref{latest}, and using
the fact that $\nu_C^{-1}(\omega)$ is the orthogonal projection
of $\nu_{D+D'}^{-1}(\omega)$  onto $\xi$, we have
\begin{align*}
V_{D+D'}(\omega) &= \frac1n \int_{\nu_{D+D'}^{-1}(\omega)}
 (y+y')\cdot \nu_{D+D'}(y+y')\, d\cH^{n-1}(y+y') \\
&=   \frac1n \int_{\nu_{D+D'}^{-1}(\omega)}
 y\cdot \nu_C({y_1}) \, d\cH^{n-1}(y+y') \\
&= \frac1n \int_{\nu_{C}^{-1}(\omega)}
 {y_1}\cdot \nu_C({y_1}) \, d\cH^{m-1}({y_1})  V_{n-m}(D') \\
&=\frac mn  V_{n-m}(D') V_C(\omega) \\
&=\frac{m(n-m)}{n^2} ra^2 \mu(\sn) \mu_\xi(\omega).
\end{align*}

Similarly, for $\omega' \subset S^{n-1}\cap \xi'$,
\[
V_{D+D'}(\omega')=\frac{m(n-m)}{n^2} ra^2 \mu(\sn) \mu_{\xi'}(\omega').
\]
Now choose $a$ so that
\[
\frac{m(n-m)}{n^2} ra^2 \mu(\sn)=1.
\]
Then
\begin{align*}
V_{D+D'}(\omega) &= \mu_\xi(\omega), \quad\text{for}\  \omega \subset
S^{n-1}\cap \xi, \\
V_{D+D'}(\omega') &= \mu_{\xi'}(\omega'), \quad\text{for}\ \omega' \subset
S^{n-1}\cap \xi'.
\end{align*}
Using  \eqref{sum-boundary} we see that
the surface area measure $S_{D+D'}$ is concentrated on
$S^{n-1}\cap(\xi\cup\xi')$,
and thus the cone-volume measure $V_{D+D'}$ is as well. From this and
the fact that
 $\mu$ is concentrated on $S^{n-1}\cap (\xi\cup \xi')$,
we conclude
\[
V_{D+D'} = \mu.
\]
\end{proof}

\begin{theo}
\label{sufficiency} Suppose $n\ge 1$ and $\mu$ is
a non-zero even finite Borel measure on $S^{n-1}$
that satisfies the subspace concentration condition. Then $\mu$
is the  cone-volume measure of
an origin-symmetric convex body in $\R^n$.
\end{theo}

\begin{proof} We first assume that strict inequality holds
 in (\ref{properdef0}) for every  linear subspace $\xi$ such that
 $0<  \dimension \xi <n$. By Theorem \ref{logmin0strict},
 the minimization problem (\ref{min-problem}) has a solution.
 Thus, by Lemma \ref{minimization-existence}, the measure $\mu$
 is a cone-volume measure of an origin-symmetric convex body in $\rn$.

Now suppose that equality holds in (\ref{properdef0}); i.e.,
\begin{equation}\label{eq1}
 \mu(\xi\cap S^{n-1}) = \frac 1n \, \mu(S^{n-1})\dimension \xi,
\end{equation}
for some linear subspace $\xi$
with  $0<\dimension \xi<n$.
In this case, from the definition of a measure that satisfies the
subspace concentration condition, there exists a linear
subspace $\xi'$  complementary to $\xi$ (obviously of dimension
$n-\dimension  \xi$) such that
\begin{equation}\label{eq2}
 \mu(\xi'\cap S^{n-1}) = \frac 1n \, \mu(S^{n-1})\dimension \xi'.
\end{equation}

If it is not the case that strict inequality holds
 in (\ref{properdef0}) for every  linear subspace $\xi$ such that
 $0<  \dimension \xi <n$, then
we proceed by induction on the dimension of the ambient space. We
start with $n=2$. In this case we must
have $\dimension  \xi = 1$ and $\mu(\xi\cap S^1)=\mu(S^1)/2$.  If we
write $\xi\cap S^1=\{u_0,-u_0 \}$ and recall that $\mu$ is even, we see
that $\mu(\{u_0\})=\mu(\{-u_0\})=\mu(S^1)/4$. Since  $\dimension  \xi = 1$, we
have $\dimension \,\xi' =1$, and if we write $\xi'\cap
S^1=\{u'_0,-u'_0 \}$, and recall that $\mu$ is even, it follows that
$\mu(\{u'_0\})=\mu(\{-u'_0\})=\mu(S^1)/4$. Thus the parallelogram (centered at
the origin) whose sides are perpendicular to $u_0, u'_0$ and whose
area is $\mu(S^1)$ is the desired convex body in $\R^2$.

We now assume that $n\ge 3$ and that the existence has been
established for all dimensions less than $n$.
From Lemma \ref{restriction}, applied to the situations described by
(\ref{eq1}) and  (\ref{eq2}),  it follows that the restrictions of
$\mu$ to
$\xi\cap S^{n-1}$ and to $\xi'\cap S^{n-1}$ satisfy the subspace
concentration conditions
in $\xi$ and $\xi'$, respectively.
Therefore,
 there exist an origin-symmetric $m$-dimensional convex body $C$ in $\xi$,
and an origin-symmetric  $(n-m)$-dimensional convex body $C'$ in $\xi'$
such that the  restrictions of $\mu$ to
$\xi\cap S^{n-1}$ and to $\xi'\cap S^{n-1}$ are
 cone-volume measures of $C$ and $C'$,
respectively.  Lemma \ref{last} will now yield the desired result.
\end{proof}

{\bf Acknowledgement} The authors thank the referees and Professor
Rolf Schneider for their very helpful comments.

\end{document}